\documentclass{amsart}
\usepackage{amssymb, amsmath}
\title{Nullstellensatz for relative existentially closed groups}
\author{M. Shahryari}
\address{M. Shahryari: Department of Mathematics, College of Science, Sultan Qaboos university, Muscat, Oman}
\email{m.ghalehlar@squ.edu.om}

\newcommand{\Rad}{\mathrm{Rad}}
\newcommand{\FX}{F_{\mathbf{V}}(X)}

\newtheorem{corollary}{Corollary}

\newtheorem {theorem}{Theorem}

\begin{document}

\maketitle
\begin{abstract}
We prove that  in every variety of $G$-groups, every $G$-existentially closed element satisfies nullstellensatz for finite consistent systems of equations. This will generalize {\bf Theorem G} of \cite{BMR1}. As a result we see that every pair of $G$-existentially closed elements in an arbitrary variety of $G$-groups generate the same quasi-variety and if both of them are $q_{\omega}$-compact, they are geometrically equivalent.
\end{abstract}


%1111111111111111111111111111111111111111111111111111111111111111111111111
%%%%%%%%%%%%%%%%%%%%%%%%%%%%%%%%%%%%%%%%%%%%%%%%%%%%%%%%%%%%%%%%%%%%%%
\vspace{3cm}

\section{Introduction}
Let $K$ be an algebraically closed field and $K[x_1, \ldots, x_m]$ be the ring of $n$-variable polynomials with coefficients in $K$. Let $S\subseteq K[x_1, \ldots, x_m]$ be a {\em system of equations} and $\Rad_K(S)$ denotes the set of all {\em consequences} of $S$ over $K$. Note that a polynomial $f$ is called a consequence of $S$, if every solution of the system $S=1$ is also a solution of $f=1$. Hilbert's classical nullstellensatz, describes the ideal $\Rad_K(S)$ in terms of elementary algebraic operation. More precisely by Hilbert's nullstellensatz, we have
$$
\Rad_K(S)=\sqrt{\langle S\rangle_{\mathrm{id}}},
$$
where $\langle S\rangle_{\mathrm{id}}$ is the ideal generated by $S$ and $\sqrt{-}$ shows the radical: the set of all polynomials $f$ such that $f^m\in \langle S\rangle_{\mathrm{id}}$ for some power $m$.

In 1999, Baumslag, Myasnikov, and Remeslennikov published the article {\em Algebraic Geometry over Groups, I. Algebraic Sets and Ideal Theory}, where the foundations of algebraic geometry of groups is introduced, \cite{BMR1}. Among many important notions discussed in their work is the notion of nullstellensatz for group equations. Generally it is any {\em effective} (algebraic or algorithmic) description of the radical of systems of equations over groups.  The last section of that article was devoted to nullstellensatz in groups. They proved the following theorem ({\bf Theorem G} of \cite{BMR1}):\\

\begin{theorem}
Let $H$ be a $G$-group and suppose that $H$ is $G$-existentially closed. Then every finite $H$-consistent system of group equations with coefficients from $G$, satisfy nullstellensatz, i.e.
$$
\Rad_H(S)= \langle S^{G[x_1, \ldots, x_m]}\rangle.
$$
\end{theorem}

In the above theorem $G$ is an arbitrary fixed group and a $G$-group is a group which contains $G$ as a distinguished subgroup. An equation with coefficients in $G$ is an element of the free product $G[x_1, \ldots, x_m]=G\ast F[x_1, \ldots, x_m]$, where $F[x_1, \ldots, x_m]$ is the free group generated by the variables $x_1, \ldots, x_n$. Being $G$-existentially closed  means that any finite system of $G$-equations and $G$-inequations which has a solution in a super-group of $H$, has already a solution in $H$.  A system is $H$-consistent if it has at least one solution in $H$.  Note that the notation
$$
\langle S^{G[x_1, \ldots, x_m]}\rangle
$$
is used for the normal closure of $S$ in the group $G[x_1, \ldots, x_m]$. There are many similarities between Theorem 1 above and the classical nullstellensatz of Hilbert, because it can be shown that in the variety of commutative rings with identity, algebraically closed fields paly the role of existentially closed structures. \\

The aim of this article is to generalize the above theorem for the case of existentially closed groups in an arbitrary variety of $G$-groups.  Our main theorem has the following formulation:\\

\begin{theorem}
Let $\mathbf{V}$ be a variety of $G$-groups and $H\in \mathbf{V}$ be a relative $G$-existentially closed element. Then for every finite $\mathbf{V}$-consistent system $S\subseteq G[x_1, \ldots, x_m]$, we have
$$
\Rad_H(S)=\langle S^{G[x_1, \ldots, x_m]}\rangle \mathrm{Id}_{\mathbf{V}}(x_1, \ldots, x_n).
$$
\end{theorem}

Recall that, in this theorem,
$$
\mathrm{Id}_{\mathbf{V}}(x_1, \ldots, x_n)\unlhd G[x_1, \ldots, x_m]
$$
consists of all $G$-identities of the variety $\mathbf{V}$, having variables $x_1, \ldots, x_n$.

\section{The proof}
In this section, we prove Theorem 2, but before, we collect some necessary notions from algebraic geometry over groups. For detailed discussion, one can consult \cite{BMR1} and \cite{MR}.

We assume that $G$ is an arbitrary group. A $G$-group is a pair $(H, \lambda)$, where $H$ is a group and $\lambda: G\to H$ is an embedding. If there is no risk of confusion, we will say that $H$ is a $G$-group, and so it contains a distinguished copy of $G$. Let $\mathcal{L}=(\cdot, ^{-1}, 1)$ be the language of groups and for any $g\in G$ attach
a constant symbol $g$ to $\mathcal{L}$. We denote the extended language by $\mathcal{L}(G)$ and so every $G$-group $H$ becomes an algebra of type $\mathcal{L}(G)$, if we interpret $g$ as $\lambda(g)$. Note that any congruence of $H$ is in fact a normal subgroup $K$ with the property $G\cap K=1$. Through this section, we will call such a normal subgroup an $G$-ideal.

For a set $X=\{ x_1, \ldots, x_n\}$ the free $G$-group generated by $X$ is the free product $G[X]=G\ast F[X]$, where $F[X]$ is the ordinary free group on $X$. We will assume that the embedding $G\hookrightarrow G[X]$ is the inclusion map. Any subset $S\subseteq G[X]$ corresponds to a  system of $G$-equations, and if $w=w(x_1,\ldots, x_n, g_1, \ldots, g_m)\in G[X]$, then the expression $w\approx 1$ is a $G$-equation with  coefficient $g_1,\ldots, g_m\in G$. Let $(H, \lambda)$ be a $G$-group. We say that $\overline{a}=(a_1, \ldots, a_n )\in H^n$ is a solution for this $G$-equation if
$$
w(a_1,\ldots, a_n, \lambda(g_1), \ldots, \lambda(g_m))=1.
$$
We denote by $V_H(S)$ the algebraic set corresponding to a system of $G$-equations $S$, which is the set of all common solutions of the elements of $S$. Define the radical of $S$, with respect to $H$, as
$$
\Rad_H(S)=\{ w\in G[X]: V_H(S)\subseteq V_H(w\approx 1)\}.
$$
This is clearly a $G$-ideal  of $G[X]$, which contains the normal closure $\langle S^{G[X]}\rangle$. The quotient group $\Gamma_H(S)=G[X]/\Rad_H(S)$ is  called the corresponding coordinate $G$-group of $S$ over $H$. Many important problems of algebraic geometry over groups transform to the study of these coordinate groups, and hence it is important to have an algebraic description of the radical $\Rad_H(S)$.

Generally there is no such an algebraic description for every system $S$, because it is equivalent to axiomatizablity of the prevariety generated by $H$ in the class of $G$-groups, see \cite{MR}. In some cases, such a description exists, and in this situation we say that a $G$-group $H$ satisfies the generalized nullstellensatz for a system $S$. Recall that a  $G$-group $H$ is $G$-existentially closed, if every finite set of $G$-equations and $G$-inequations with a solution in a larger $G$-group $H^{\prime}$ has already a solution in $H$. As we said in the introduction, Baumslag, Myasnikov, and Remeslennikov, proved that $G$-existentially closed $G$-groups satisfy nullstellensatz for finite consistent systems. We can generalize their result for the case of relatively existentially closed elements of varieties. If $\mathbf{V}$ is a variety of $G$-groups, then an element $H\in \mathbf{V}$ is relative $G$-existentially closed,  if every finite set of  $G$-equations and $G$-inequations with a solution in a larger $G$-group $H^{\prime}\in \mathbf{V}$ has already a solution in $H$. Clearly, this is equivalent to the existence of solution in $H$ for every $\mathbf{V}$-consistent finite system of $G$-equations and $G$-inequations (systems with at least one solution in some element of $\mathbf{V}$). It is proved that every element of $\mathbf{V}$ embeds in some relative $G$-existentially closed $G$-group $H\in \mathbf{V}$, and this $H$ is not too large (see for example \cite{Shah}). Now, we are ready to give the proof of Theorem 2.

\begin{proof}
Suppose $X=\{ x_1, \ldots, x_n\}$. We denote the free $G$-group of the variety $\mathbf{V}$ over the set $X$ by $\FX$. Let $T$ be the image of $S$ in $\FX$, that is the set $S$ modulo $G$-identities of the variety $\mathbf{V}$. Suppose also $\Rad_H^{\mathbf{V}}(T)$ is relative radical of $T$ in $\mathbf{V}$. Then we have $\Rad_H^{\mathbf{V}}(T)$ is a $G$-ideal of $\FX$. Let $Q=\langle T^{\FX}\rangle$ be the normal closure of $T$ in $\FX$. Note that $V_H(Q)\neq \emptyset$, so $Q\cap G=1$. Hence $K=\FX/Q$ is a $G$-group and it belongs to $\mathbf{V}$. Let $T=\{ w_1, \ldots, w_k\}$ and suppose $f$ does not belong to $Q$. Consider the finite system of $G$-equations and $G$-inequations $T^{\prime}$:
$$
w_1\approx 1, \ldots, w_k\approx 1, f\not\approx 1.
$$
Let $H^{\prime}=H\times K$ and $b_1=x_1Q, \ldots, b_n=x_nQ$. The for all $1\leq i\leq k$, we have
$$
w_i(b_1, \ldots, b_n)=w_i(x_1, \ldots, x_n)Q=Q (=1\ in\ K),
$$
and on the other hand
$$
f(b_1, \ldots, b_n)=fQ\neq Q.
$$
This shows that $(b_1, \ldots, b_n)$ is a solution of $T^{\prime}$ in $H^{\prime}$, therefore $T^{\prime}$ has a solution in $H$, say $a_1, \ldots, a_n\in H$. This means that $(a_1, \ldots, a_n)\in V_H(T)$, but since $f(a_1, \ldots, a_n)\neq 1$, so $f$ does not belong to $\Rad_H^{\mathbf{V}}(T)$. This shows that $\Rad_H^{\mathbf{V}}(T)\subseteq Q$ and hence $ Q=\Rad_H^{\mathbf{V}}(T)$. Now, we claim that
$$
\Rad_H^{\mathbf{V}}(T)=\frac{\Rad_H(S).\mathrm{Id}_{\mathbf{V}}(X)}{\mathrm{Id}_{\mathbf{V}}(X)}.
$$
Note that $\FX=G[X]/I$, where $I$ stands for $\mathrm{Id}_{\mathbf{V}}(X)$. Note that also
$$
T=\{ wI:\ w\in S\}.
$$
Since for all $(h_1, \ldots, h_n)\in H^n$, we have
$$
(wI)(h_1, \ldots, h_n)=w(h_1, \ldots, h_n),
$$
so, we have $V_H(T)=V_H(S)$. Therefore
\begin{eqnarray*}
\Rad_H^{\mathbf{V}}(T)&=&\Rad_H^{\mathbf{V}}(V_H(T))\\
                      &=&\Rad_H^{\mathbf{V}}(V_H(S))\\
                      &=&\{wI\in \FX: V_H(S)\subseteq V_H(wI\approx 1)\}\\
                      &=&\{wI\in \FX: V_H(S)\subseteq V_H(w\approx 1)\}\\
                      &=&\{wI\in \FX: w\in \Rad_H(S)\}\\
                      &=&\frac{\Rad_H(S).I}{I}.
\end{eqnarray*}
Note that the normal closure $\langle T^{\FX}\rangle$ is generated by the elements of the form $(uI)(wI)^{\pm1}(u^{-1}I)$, with $u\in G[X]$ and $w\in S$. Therefore, we have
$$
\langle T^{\FX}\rangle=\frac{\langle S^{G[X]}\rangle.I}{I}.
$$
Comparing these two last equalities, we obtain
$$
\Rad_H(S)=\langle S^{G[X]}\rangle. \mathrm{Id}_{\mathbf{V}}(X).
$$
\end{proof}

A $G$-group $H$ is called $q_{\omega}$-compact, if for every system $S$, we have
$$
\Rad_H(S)=\bigcup\{ \Rad_H(S_0): S_0\subseteq S, \ S_0=finite\}.
$$
Being $q_{\omega}$-compact is a generalization of the property of being $G$-equational Noetherian (see \cite{MR} for its various equivalents and logical implications).
We say that two $G$-groups $H$ and $H^{\prime}$ are geometrically equivalent, if $\Rad_H(S)=\Rad_{H^{\prime}}(S)$, for all system $S$. In this case the algebraic geometry of $H$ and $H^{\prime}$ are essentially the same. As a result, we have

\begin{corollary}
Let $\mathbf{V}$ be a variety of $G$-groups and $H$ and $H^{\prime}$ be two relative $G$-existentially closed elements of $\mathbf{V}$. If $H$ and $H^{\prime}$ are $q_{\omega}$-compact, then they are geometrically equivalent.
\end{corollary}

\begin{proof}
First suppose that $S\subseteq G[X]$ is a consistent system of $G$-equations. Then, we have
\begin{eqnarray*}
\Rad_H(S)&=&\bigcup\{ \Rad_H(S_0): S_0\subseteq S, \ S_0=finite\}\\
         &=&\bigcup\{ \langle S_0^{G[X]}\rangle.\mathrm{Id}_{\mathbf{V}}(X): S_0\subseteq S, \ S_0=finite\}\\
         &=&\bigcup\{ \Rad_{H^{\prime}}(S_0): S_0\subseteq S, \ S_0=finite\}\\
         &=&\Rad_{H^{\prime}}(S).
\end{eqnarray*}
If $S$ is not consistent, then we have $V_H(S)=V_{H^{\prime}}(S)=\emptyset$, and hence $\Rad_H(S)=\Rad_{H^{\prime}}(S)$.
\end{proof}

In the special case of $G=1$, every element of the variety  of metabelian groups or the variety of nilpotent groups of a fixed class is equationally Noetherian and hence is $q_{\omega}$-compact. Therefore, we have the next result:

\begin{corollary}
Let $H$ and $H^{\prime}$ be two relative existentially closed groups in the variety of metabelian groups or in the variety of nilpotent groups of the fixed class $c$. Then $H$ and $H^{\prime}$ are geometrically equivalent as $1$-groups.
\end{corollary}

\end{document}